\documentclass[11pt]{amsart}

\usepackage{amsmath,amsthm,amssymb,amsfonts,ifpdf}
\usepackage{comment}
\usepackage{xcolor}
\usepackage{tikz}
\usepackage{pgfplots}

\usepackage{xargs}  
\usepackage[colorinlistoftodos,prependcaption,textsize=tiny,obeyFinal]{todonotes}
\newcommandx{\ebltodo}[2][1=]{\todo[linecolor=red,backgroundcolor=red!25,bordercolor=red,#1]{#2}}
{
  \color{olive}%
}%
{}

\ifpdf
\usepackage[pdftex,pdfstartview=FitH,pdfborderstyle={/S/B/W 1},%
colorlinks=true, linkcolor=blue, urlcolor=blue, citecolor=blue,%
pagebackref=true]{hyperref}
\else
\usepackage[hypertex]{hyperref}
\fi
\usepackage{enumitem}
\setenumerate{leftmargin=*}
\setitemize{leftmargin=*}
\usepackage{amscd}
\usepackage[latin1]{inputenc}
\DeclareMathAlphabet{\mathpzc}{OT1}{pzc}{m}{it}
\usepackage{bbm}

\numberwithin{equation}{section}

\swapnumbers
\newtheorem{thm}[subsection]{Theorem}

\newtheorem*{cor*}{Corollary}
\newtheorem{lemma}[subsection]{Lemma}

\newtheorem{propos}[subsection]{Proposition}

\newtheorem*{thm*}{Theorem}
\newtheorem*{thma*}{Theorem A}
\newtheorem*{thmb*}{Theorem B}
\newtheorem*{thmc*}{Theorem C}

\swapnumbers
\theoremstyle{definition}

\newtheorem{remark}[subsection]{Remark}

\newcounter{consta}

\newcounter{constk}

\newcounter{constc}

\newcounter{constE}

\newcounter{constd}

\makeatletter
\newcommand*\bigcdot{\mathpalette\bigcdot@{.5}}
\newcommand*\bigcdot@[2]{\mathbin{\vcenter{\hbox{\scalebox{#2}{$\m@th#1\bullet$}}}}}
\makeatother

\def\XXint#1#2#3{{\setbox0=\hbox{$#1{#2#3}{\int}$ }
\vcenter{\hbox{$#2#3$ }}\kern-.6\wd0}}

\DeclareMathOperator{\Leb}{Leb}
\DeclareMathOperator{\diam}{diam}
\DeclareMathOperator{\diff}{d}

\DeclareMathOperator\Mat{Mat}

\newcommand\SL{{\rm{SL}}}

\newcommand\SO{{\rm{SO}}}
\newcommand\Lie{{\rm Lie}}

\def\sl{{\mathfrak{sl}}}

\newcommand\Hom{{\rm{Hom}}}

\def\bbr{\mathbb{R}}

\def\bbc{\mathbb{C}}

\def\bbn{\mathbb{N}}

\def\R{\bbr}
\def\C{\bbc}

\def\N{\bbn}

\def\pfrak{\mathfrak{P}}

\def\pfrak{\mathfrak{p}}

\def\rfrak{\mathfrak{r}}

\def\tbf{\mathbf{t}}

\def\wbf{\mathbf{w}}

\def\vare{\varepsilon}

\def\zg0{Z_{G_\omega}(s)}

\def\zg{Z_G(s)}

\def\be{\begin{equation}}
\def\ee{\end{equation}}

\newcommand{\rhsc}{\delta}

\newcommand {\absolute}[1] {\left| {#1} \right|}
\newcommand {\norm}[1] {\left\| {#1} \right\|}

\newcommand\eng{\mathcal E}
\newcommand\egbd{C}

\newcommand{\hide}[1]{}

\newcommand{\sqf}{Q_0}

\title{Projection Theorems in the Presence of Expansions}

\address{Department of Mathematics, University of California, Berkeley, CA 94720}
\email{zuo\_lin@berkeley.edu}
\thanks{}

\author{Z.\ Lin}

\author{K.\ W.\ Ohm}
\date{}

\address{Department of Mathematics, University of California, Berkeley, CA 94720}
\email{kwohm@berkeley.edu}
\thanks{}

\begin{document}

\begin{abstract}
    We prove a restricted projection theorem for a certain one dimensional family of projections from $\R^n$ to $\R^k$. 
    
    The family we consider here arises naturally in the study of quantitative equidistribution problems in homogeneous dynamics.
\end{abstract}

\maketitle



\setcounter{tocdepth}{1}
\tableofcontents

\section{Introduction}\label{sec: introduction}
Restricted projection problems are intimately related to central questions in Fourier analysis and incidence geometry,
and have been much studied, e.g., by Mattila, Falconer, Bourgain and others. More recently, certain restricted projection theorems have also found striking applications in homogeneous dynamics. 

Let us recall the classical Marstrand projection theorem: 
Let $K\subset\R^n$ be a compact subset, then for a.e.\ $v\in\mathbb S^{n-1}$
\be\label{eq: Marstrand}
\dim{\rm p}_{v}(K)=\min(1,\dim K),  
\ee
where ${\rm p}_v(w)=w\cdot v$ is the orthogonal projection in the direction of $v$ and here and in what follows $\dim$ denotes the Hausdorff dimension. Analogous statements hold more generally for orthogonal projection into a.e.\ $m$-dimensional subspace, with respect to the Lebesgue measure on ${\rm Gr}(m, n)$.  

Broadly speaking, restricted projection problems seek to obtain similar results as in~\eqref{eq: Marstrand} where $v$ is confined to a proper Borel subset $\mathsf B\subset \mathbb S^{n-1}$. 
Note, however, that without further restrictions on $\mathsf B$,~\eqref{eq: Marstrand} fails: e.g., if 
\begin{align*}
    \mathsf B=\{(\cos t, \sin t, 0): 0\leq t\leq 2\pi\}
\end{align*}
is the great circle in $\mathbb S^2$ and $K$ is the $z$-axis, then ${\rm p}_v(K)=0$ for every $v\in \mathsf B$. 

To avoid degenerations of this nature, one may, e.g., consider a curve $\gamma:[0,1]\to\R^n$ satisfying that 
\begin{align*}
    \{\gamma'(r), \gamma''(r),\ldots, \gamma^{(n)}(r)\}\qquad\text{spans $\R^n$ for all $r\in[0,1]$.}
\end{align*}
It is natural to inquire whether for almost all $r\in[0,1]$ the following holds
\be\label{eq: Rest Proj Intr}
\dim{\rm p}_{\gamma'(r)}(K)=\min(1,\dim K);
\ee
again $K\subset\R^n$ is a compact subset. 

Indeed~\eqref{eq: Rest Proj Intr} was conjectured by F\"assler and Orponen~\cite{FaOr} in dimension $3$. 
This conjecture was resolved by K\"aenm\"aki, Orponen, and Venieri \cite{kenmki2017marstrandtype} and Pramanik, Yang, and Zahl~\cite{PYZ} --- these works rely on the work of Wolff and Schlag on circular Kakeya sets~\cite{Wolff,Schlag}.

More recently, Gan, Guo, and Wang~\cite{GGW} have established~\eqref{eq: Rest Proj Intr} in all dimensions using decoupling inequalities for the moment curve due to Bourgain, Demeter, and Guth~\cite{BourDemGuth}.   

\smallskip

In this paper, we study a closely related family of projections. 
These families are mainly motivated by the aforementioned applications to homogeneous dynamics.

Indeed we first study the problem in dimension $3$, where our argument inspired is by~\cite{OV-Planes}. 
Then we use the main results in~\cite{GGW} to study a related, but more localized version, of this problem in all dimensions.   
    
Let us fix some notation in order to state the results of this paper. For every Borel subset of $\mathsf B\subset \R^d$, we denote the Lebesgue measure of $\mathsf B$ by $|\mathsf B|$. For every finite set $F$, we denote the cardinality of $F$ by $\#F$. We refer to Subsection~\ref{subsec:notations} for the convention on $\ll$ and $\star$ notations. 

We define the following family of projections with expansion. For every $t\geq 0$ and $r\in[0,1]$, let $\pi_{t,r}:\R^3\to \R^2$ be 
\be\label{eq: def pi r dim 3}
\pi_{t,r}(x,y,z)=\Bigl(e^t(x+ry+\tfrac{r^2}{2}z), y+rz\Bigr).
\ee

The following is one of the results of this paper.     

\begin{thm}\label{thm: main finitary}
    Let $1\leq \alpha\leq 3/2$, and let  $0<\delta_0\leq1$. 
Let $F\subset B_{\R^3}(0,1)$ be a finite set and let $\mu$ denote the uniform measure on $F$. Assume that 
\begin{align*}
    \mu(B_{\R^3}(w, \delta))\leq C_0\cdot \delta^\alpha\quad\text{for all $w\in \R^3$ and all $\delta\geq \delta_0$}
\end{align*}
where $C_0\geq 1$.

Let $0<\vare<\alpha/100$. For all $\delta_0 \leq \delta \ll_{\vare} 1$ and $t > 0$ satisfying $\delta \geq e^t\rhsc_0$, there exists $\mathcal{E} = \mathcal{E}_{{t}, \delta} \subset [0,1]$ with $|\mathcal{E}|\ll_\vare \delta^{\star\vare}$ so that the following holds. 

Let $r\in [0,1]\setminus \mathcal E$, then there exists $F_{t,\delta,r}\subset F$ with 
\begin{align*}
    \mu(F\setminus F_{t,\delta,r})\ll_\vare \delta^{\star\vare}
\end{align*}
such that for all $w\in F_{t,\delta,r}$, we have 
\begin{align*}
\mu\Bigl(\{w'\in F_{t, \delta, r}: \|\pi_{t,r}(w')-\pi_{t,r}(w)\|\leq \delta\}\Bigr)\ll_\vare
    C_0 e^{-t/10}\delta^{\alpha-\sqrt{\vare}}.
\end{align*}
\end{thm}

\begin{remark}
We remark that for the $\sqrt{\epsilon}$-loss, we prove in fact a slightly improved version $e^{-t/10}\delta^{\alpha-\vare}$ for $t$ large. See the proof of Theorem~\ref{thm: main finitary} in Section~\ref{sec: proof of main 1} for more detail. 
\end{remark}

We emphasize that the exponential saving in this theorem $e^{-t/10}$ is fixed through the interval $\alpha\in[1,3/2]$. If one is content with a saving $e^{(1-\alpha)t}$, which indeed is lost as $\alpha$ tends to $1$, Theorem~\ref{thm: main finitary} can be proved using the main result in~\cite{OV-Planes} and a simple covering argument for $\alpha \in [1, 2]$. This was pointed out by an anonymous referee; this argument also exists in an earlier version of this paper which is available on the arXiv. 

As mentioned earlier, the proof of Theorem~\ref{thm: main finitary} is based on arguments inspired by~\cite{OV-Planes}, together with~\cite{kenmki2017marstrandtype,PYZ} which in turn rely based on Wolff's work~\cite{Wolff,Schlag, Zahl}. 

We now present a similar but different general result on dimension improvement
in irreducible $\SL_2(\R)$-representations. In dimensions $3, 5, 7$, and $11$,
results of this nature drive the Margulis function estimate in~\cite{LMWY}; see
also~\cite{LMW22}. The proof relies on deep ingredients; in particular, our
argument makes crucial use of~\cite[Thm.~2.1]{GGW}. We refer to
Remark~\ref{rem:comparison on stable direction} for a discussion of the
relationship between these two results, and mention here only that, in spite of
its limitations, the approach we employ for
Theorem~\ref{thm: main finitary} is more elementary, and the boxes it involves
are more global: they are localized only in the unstable and central directions. 

Let $V$ be an irreducible representation of $\SL_2(\R)$ with $\dim V = n + 1$. We fix a norm on $V$. Note that since all norms on a finite-dimensional vector space are equivalent, different choices of norms will only create a constant factor in the dimension estimate. Let
\begin{align*}
    u_r = \begin{pmatrix}
        1 & r\\
         & 1
    \end{pmatrix}
\end{align*}
and 
\begin{align*}
    a_{t} = \begin{pmatrix}
        e^{{t}/2} & \\
         & e^{-{t}/2}
    \end{pmatrix}.
\end{align*}

We have the following result on  dimension improvement under the push by $a_{t} u_r$ for a discretized measure in an irreducible representation of $\SL_2(\R)$. 

\begin{thm}\label{thm:Improving Dimension SL2 irrrep}
    Let $\alpha \in (0, n + 1)$, $C \geq 1$ and $0 < \delta_0 \leq 1$. Let $F$ be a finite set in $B_{V}(0, 1)$ and let $\mu$ denote the uniform probability measure on $F$. Assume that 
    \begin{align*}
        \mu(B_{V}(x, \delta)) \leq C \delta^{\alpha}
    \end{align*}
    for all $\delta \in [\delta_0, 1]$. Then the following holds for all $\epsilon \in (0, \frac{1}{10^4 n}\min\{\alpha, n + 1 - \alpha\})$.

    For all $\delta \in [\delta_0, \frac{1}{100}]$ and ${t} \gg_{\epsilon} 1$ with $e^{\frac{n}{2}{t}}\delta_0 \leq \delta \leq e^{-\frac{n}{2}{t}}$, there exists $\mathcal{E} = \mathcal{E}_{{t}, \delta} \subset [0,1]$ with $|\mathcal{E}| \ll_{\alpha, \epsilon} e^{-\star\epsilon^2{t}}$ satisfying the following. 
    
    For all $r \in [0, 1] \setminus \mathcal{E}$, there exists $F_{t, \delta, r} \subset F$ with $\mu(F \setminus F_{t, \delta, r}) \ll_{\alpha, \epsilon} e^{-\star\epsilon^2{t}}$ so that for all $w \in F_{t, \delta, r}$, we have
    \begin{align*}
        \mu(\{w' \in F_{t, \delta, r}: \|a_t u_r(w') - a_t u_r(w)\| \leq \delta\}) \ll_{\epsilon} C e^{-\frac{\varpi(\alpha){t}}{2}} \delta^{\alpha - \epsilon}
    \end{align*}
    where
    \begin{align*}
        \varpi(\alpha) = \max\{n\alpha - \lfloor\alpha\rfloor(\lfloor\alpha\rfloor + 1), \lfloor\alpha\rfloor(2n - \lfloor\alpha\rfloor + 1) - n\alpha\}.
    \end{align*}
\end{thm}

\begin{remark}
    Note that if $\alpha$ varies in compact subsets of $(0, n + 1)$, $\varpi(\alpha)$ is uniformly bounded away from $0$.  
\end{remark}

\begin{remark}
When $n = 2$, the representation is the adjoint representation of $\SL_2(\R)$. We have
\begin{align*}
\frac{1}{2}\varpi(\alpha) = \begin{cases}
\alpha & 0 < \alpha \leq 1,\\
\max\{2 - \alpha, \alpha - 1\} & 1 < \alpha \leq 2,\\
3 - \alpha & 2 < \alpha \leq 3.
\end{cases}
\end{align*}
See Figure~\ref{fig: improve dim 3}.  
\end{remark}

\begin{figure}[htbp]\label{fig: improve dim 3}
\centering
\begin{tikzpicture}
\begin{axis}[axis lines = middle, xmin=0, xmax=3, ymin=0, ymax=1, samples=100, 
  ytick={1}, yticklabels={$1$}, enlargelimits, 
  height=6cm, width=12cm, axis line style={thick}, xtick=\empty, yticklabels={}]
  
  \addplot[domain=0:1, thick] {x};
  
  \addplot[domain=1:2, thick] {max(2-x, x-1)};
  
  \addplot[domain=2:3, thick] {3-x};

  \node at (axis cs:2.93,0) [anchor=north west] {$3$};
  
  \node at (axis cs:0,1) [anchor=east] {$1$};

\end{axis}
\end{tikzpicture}
\caption{Graph of the piecewise function \( \varpi(\cdot) \)}
\end{figure}

\begin{remark}\label{rem:comparison on stable direction}
    We remark that Theorem~\ref{thm:Improving Dimension SL2 irrrep} is a generalization of Theorem~\ref{thm: main finitary} when ${t} = |\log \delta|$. In this case, the pre-image of a $\delta$-disk under the expanded projections in Theorem~\ref{thm: main finitary} is of size $\delta^2 \times \delta \times 1$, which coincides with the pre-image of a $\delta$-ball under $a_{t} u_r$. 

    For general $t$, Theorem~\ref{thm:Improving Dimension SL2 irrrep} does not imply Theorem~\ref{thm: main finitary} due to the difference in the stable direction. Note that from the 'projection and slicing' picture, the preimage of a $\delta$-ball under the \emph{projection} $\pi_{t, r}$ is a rectangle of size $e^{-t} \delta \times \delta \times 1$ where the preimage of a $\delta$-ball under the \emph{action} of $a_t u_r$ is a rectangle of size $e^{-t} \delta \times \delta \times e^t \delta$.  
\end{remark}

\subsection{Constants and \texorpdfstring{$\star$-notations}{⋆-notations}}\label{subsec:notations}
Throughout the paper, the notation $a\ll b$ and $a^{\star b}$ mean $a\leq D b$ and $a^{Db}$, respectively, where $D$ is some positive constant depending only on the ambient norm vector space $(\R^{n + 1}, \|\cdot\|)$. For $A \ll_E B$, we mean there exist constant $D_E > 0$ depending on $E$ and the ambient norm vector space $(\R^{n + 1}, \|\cdot\|)$ so that $A \leq D_E B$. For $A \asymp B$, we mean $A \ll B$ and $B \ll A$. 




\section{Proof of Theorem~\ref{thm: main finitary}}\label{sec: proof of main 1}
In this section we will prove Theorem~\ref{thm: main finitary}. 
Let us begin by fixing some notation which will be used throughout this section. 

Recall that for all finite set $F$ we use $\mu$ to denote the uniform measure on $F$. That is: 
\begin{align*}
    \mu(\mathsf B)=\frac{\#(\mathsf B\cap F)}{\#F}\qquad\text{for any Borel set $\mathsf B\subset \R^3$}.
\end{align*}

For every $w\in F$, all $r\in [0,1]$, and all $b>0$, let 
\begin{align*}
    m^b(\pi_{t, r}(w))=\mu(\{w'\in F: \|\pi_{t, r}(w)-\pi_{t, r}(w')\|\leq b\}).
\end{align*}
More generally, given a subset $\mathsf B\subset \R^3$, let 
\begin{align*}
m^b(\pi_{t, r}(w)|\mathsf B)=\mu(\{w'\in \mathsf B: \|\pi_{t, r}(w)-\pi_{t, r}(w')\|\leq b\}).
\end{align*}
For all $w\in F$ and all $b>0$, let 
\begin{align*}
    \mathsf D_b(w)=\{w':b\leq \|w-w'\|\leq 2b\}.
\end{align*}

The theorem in the \emph{$t$ small} range $e^{-t}\geq \delta^{\sqrt\vare}$ follows from the following lemma. 

\begin{lemma}[\cite{OV-Planes}, Lemma 2.3]\label{lem: main OV}
The following holds for all small enough $\eta$ and all large enough $C$. Suppose $\delta_0 \leq \delta \ll_{\eta} 1$. 
Let $F_{\rm hm}\subset F$ denote the set of $w\in F$ so that  
\begin{align*}
    |\{r\in[0,1]: m^\delta(\pi_{0, r}(w))\geq C_0\delta^{\alpha-8\eta}\}|\geq C\delta^{\eta}.
\end{align*}
Then $\mu(F_{\rm hm})\leq C \delta^{\eta}$. 
\end{lemma}

\begin{proof}
The proof is essentially the same as \cite[Lemma 2.3]{OV-Planes}. We sketch the needed changes for the convenience of the reader. 

Define $\Phi: \R^2 \times \R^2$ as
\begin{align}\label{eqn:Phi circle}
    \Phi(x, y) = x_2 + 2x_1y_1 + y_2x_1^2.
\end{align}
A direct computation shows that $\Phi$ satisfies the cinematic curvature condition as in \cite[Eqn.~(14, 15)]{KW99}, see also \cite[Section 4]{Zahl}. The arguments in the proof of \cite[Lemma 2.3]{OV-Planes} applies if we replace the Marstrand's three circles lemma by \cite[Lemma 3.1]{KW99}, see also Lemma~\ref{lem: Quadratic formula} and the remark after the proof of it. 
\end{proof}
 
Similarly, the following lemma, which is analogous to Lemma~\ref{lem: main OV} in the range $t$ is {\em large}, i.e., $e^{-t}\leq \delta^{\sqrt\eta}$, 
is the main step in the proof of Theorem~\ref{thm: main finitary}.  

\begin{lemma}\label{lem: main 1st}
The following holds for all small enough $\eta$ and all large enough $C$.
Assume $\delta$ and $t$ satisfies $e^{t} \delta_0 \leq \delta \ll_{\eta} 1$ and $e^{-t}\leq \delta^{\sqrt\eta}$. 
Let $F_{\rm hm}\subset F$ denote the set of $w\in F$ so that  
\begin{align*}
    |\{r\in[0,1]: m^\delta(\pi_{t, r}(w))\geq C_0e^{-t/10}\delta^{\alpha-8\eta}\}|\geq C\delta^{\eta}.
\end{align*}
Then $\mu(F_{\rm hm})\leq C \delta^{\eta}$. 
\end{lemma}

The extra factor $e^{-t/10}$ causes significant difficulties for adapting the proof in~\cite{OV-Planes}. 
One of the main issues here is to prove (an improved) version of the fact that the high multiplicity at $\pi_{t, r}(w)$ is not localized, see Lemma~\ref{lem: small radius}.   

We also remark that the constant $8$ appearing in Lemma~\ref{lem: main OV} and Lemma~\ref{lem: main 1st} is not explicated in the proof in \cite{OV-Planes}, however, as we will see in the proof of Lemma~\ref{lem: main 1st} this is what the proof gives. 

We first deduce Theorem~\ref{thm: main finitary} from these two lemma. 
\begin{proof}[Proof of Theorem~\ref{thm: main finitary}]
The proof of Theorem~\ref{thm: main finitary} is based on the above lemmas and Fubini's theorem as we now explicate.  

Let $\vare \in (0, \alpha/100)$ and let $\eta=\vare/20$. For all $r\in[0,1]$, let 
\begin{align*}
F_{\rm bad}(r)=\{w: m^\delta(\pi_{t, r}(w))\geq C_0e^{-t/10}\delta^{\alpha-8\eta}\}.
\end{align*}
The claim in Theorem~\ref{thm: main finitary} follows if we show that there is a subset $I_\delta\subset [0,1]$ with $|[0,1]\setminus I_\delta|\ll\delta^{\eta/2}$
so that for all $r\in I_\delta$, we have 
\begin{align*}
\mu(F_{\rm bad}(r))\ll\delta^{\eta/2}.
\end{align*}

Let $C > 1$ and let $I_{\rm bad} \subseteq [0, 1]$ be the set consisting of $r$ so that 
\begin{align*}
    \mu(F_{\rm bad}(r)) \geq C \delta^{\eta/2}.
\end{align*}
Assume $|I_{\rm bad}| \geq C\delta^{\eta/2}$, we will show this lead to a contradiction provided $C$ is large enough. 

Equip $[0,1]\times F$ with the product measure $\Leb\times \mu$. Let 
\begin{align*}
\mathsf E=\{(r,w)\in[0,1]\times F: m^\delta(\pi_{t, r}(w))\geq C_0e^{-t/10}\delta^{\alpha-8\eta}\},
\end{align*}
and for every $w\in F$, let $\mathsf E_w=\{r: (r,w)\in\mathsf E\}$. 
The above then implies that 
\begin{align*}
\Leb\times \mu(\mathsf E)\geq C^2\delta^\eta.
\end{align*}
Set $F'=\{w\in F: |\mathsf E_w|\geq C \delta^\eta\}$. Then 
using Fubini's theorem, we conclude   
\begin{align}\label{eqn:fubini1}
    \mu(F')\geq \tfrac12C^2\delta^{\eta}.
\end{align}
Moreover, in view of the definitions, for every $w\in F'$
\begin{align}\label{eqn:fubini2}
    |\{r\in[0, 1]: m^\delta(\pi_{t, r}(w))\geq C_0e^{-t/10}\delta^{\alpha-8\eta}\}|\geq C\delta^{\eta}.
\end{align}

We start with \emph{$t$ large} range, i.e., $e^{-t}\leq \delta^{\sqrt\eta}$. In this range Eqn.~(\ref{eqn:fubini1}) and (\ref{eqn:fubini2}) contradict Lemma~\ref{lem: main 1st} provided $C$ is large enough. 

We now focus on the \emph{$t$ small} range, i.e., $e^{-t}\geq \delta^{\sqrt\eta}$. In this range the theorem follows from the case where $t = 0$. If $t = 0$, Eqn.~(\ref{eqn:fubini1}) and (\ref{eqn:fubini2}) contradict Lemma~\ref{lem: main OV} provided $C$ is large enough. This proves the theorem in this case. For general $t$ so that $e^{-t}\geq \delta^{\sqrt\eta}$, we let $\mathcal{E}_{t, \delta} = \mathcal{E}_{0, \delta}$. We have $|\mathcal{E}_{t, \delta}| = |\mathcal{E}_{0, \delta}| \ll_{\epsilon} \delta^{\star \epsilon}$ and for all $r \in [0, 1] \setminus \mathcal{E}_{t, \delta}$, let $F_{t, \delta, r} = F_{0, \delta, r}$, we have
\begin{align*}
    \mu(F \setminus F_{t, \delta, r}) = \mu(F \setminus F_{0, \delta, r}) \ll_{\epsilon} \delta^{\star\epsilon}.
\end{align*}
Moreover, for all $w \in F_{t, \delta, r}$, we have
\begin{align*}
    &\mu(\{w' \in F_{t, \delta, r}:\|\pi_{t, r}(w') - \pi_{t, r}(w)\| \leq \delta\})\\
    \leq{}& \mu(\{w' \in F_{t, \delta, r}:\|\pi_{0, r}(w') - \pi_{0, r}(w)\| \leq \delta\}) \ll_{\epsilon}  C_0\delta^{\alpha - \epsilon}.
\end{align*}

In conclusion, we showed that there exists $\mathcal{E} = \mathcal{E}_{{t}, \delta} \subset [0,1]$ with $|\mathcal{E}|\ll_\vare \delta^{\star\vare}$ so that the following holds. Let $r\in [0,1]\setminus \mathcal E$, then there exists $F_{t,\delta,r}\subset F$ with 
\begin{align*}
\mu(F\setminus F_{t,\delta,r})\ll_\vare \delta^{\star\vare}
\end{align*}
such that for all $w\in F_{t,\delta,r}$, we have 
\begin{align*}
\mu\Bigl(\{w'\in F_{t, \delta, r}: \|\pi_{t,r}(w')-\pi_{t,r}(w)\|\leq \delta\}\Bigr)\ll_\vare \begin{cases}
    C_0 e^{-t/10}\delta^{\alpha-\vare} & \text{ if $e^{-t} \leq \delta^{\sqrt{\frac{\epsilon}{20}}}$},\\
    C_0 \delta^{\alpha-\vare} & \text{ if $e^{-t} \geq \delta^{\sqrt{\frac{\epsilon}{20}}}$}.
\end{cases}
\end{align*}
Note that if $\epsilon \leq \frac{1}{100}$, we have $\delta^{-\frac{9}{10}\sqrt{\epsilon}} \geq \delta^{-\epsilon}$. Therefore
\begin{align*}
\mu\Bigl(\{w'\in F_{t, \delta, r}: \|\pi_{t,r}(w')-\pi_{t,r}(w)\|\leq \delta\}\Bigr)\ll_\vare    C_0 e^{-t/10}\delta^{\alpha-\sqrt{\vare}}.
\end{align*}
The proof of the theorem is complete. 
\end{proof}

\medskip

We now focus on Lemma~\ref{lem: main 1st}. The following projection theorem is a key ingredient of the proof. 

\begin{thm}\label{thm: one dim optimal}
Let $\sigma$ be a probability measure on $B(0,10)$ which satisfies 
\begin{align*}
\sigma(B(w,r))\leq \hat Cr^\beta\quad\text{for all $w\in\R^3$ and all $r\geq r_0$}
\end{align*}
where $1/2\leq \beta\leq 1$ and $\hat C>0$. 

The following holds for all $\vare_0$ small enough.
For every $b\geq r_0$, there exists a subset $E_b$ with $\sigma(B(0,10)\setminus E_b)\leq b^{\vare_0}$,
and for every $z\in E_b$, there is a subset $I_z\subset [0,1]$ with $|[0,1]\setminus I_z|\leq b^{\vare_0}$ so that for every $r\in I_z$
\begin{align*}
\sigma\{z': |(1,r,\tfrac{r^2}{2})\cdot (z-z')|\leq b\}\ll_{\epsilon_0} \hat C b^{\beta-\vare_0}.
\end{align*}
\end{thm}

\begin{proof}
This theorem is a consequence of recent developments in restriction projection theory, see~\cite{kenmki2017marstrandtype} and~\cite{PYZ}. See also~\cite[Thm.~ B.1]{LM-PolyDensity} for this particular formulation.  
\end{proof}

In the proof of Lemma~\ref{lem: small radius} below, we will apply Theorem~\ref{thm: one dim optimal} with $\beta=1$. It is worth mentioning that the proof of Lemma~\ref{lem: small radius} is similar to the proof of Theorem~\ref{thm:Improving Dimension SL2 irrrep} for $n=2$, albeit we use Theorem~\ref{thm: one dim optimal} here instead of~\cite[Thm.~2.1]{GGW} which is used there.

%

\medskip

Since $t$ is fixed so that $e^{-t} \leq \delta^{\sqrt{\epsilon}}$ throughout the following argument, we will write $\pi_r$ for $\pi_{t,r}$ in the rest of this section. 

\begin{lemma}\label{lem: small radius}
Let the notation and assumptions be as in Lemma~\ref{lem: main 1st}. 
Assume $\mu(F_{\rm hm})\geq C\delta^{\eta}$. 
The following holds for all small enough $\eta$ and large enough $C$. 
There exists some $b\geq e^{t}\delta^{1-3\eta}$ and a subset $F'_{\rm hm}\subset F_{\rm hm}$ with 
\be\label{eq: F' hm}
\mu(F'_{\rm hm})\gg C \delta^{\frac{11\eta}{10}}
\ee
so that for all $w\in F'_{\rm hm}$, we have  
\be\label{eq: r for F'}
|\{r\in [0,1]: m^\delta(\pi_r(w)|\mathsf D_b(w))\geq C_0e^{-t/10}\delta^{\alpha-7\eta}\}|\geq \delta^{\frac{11\eta}{10}}.
\ee
\end{lemma}

\begin{proof}
Suppose $\delta \geq e^{t}\delta_0$ and $\delta$ is small enough depending only on $\eta$ so that $\delta^{\frac{\eta}{10}} \leq |\log \delta|^{-1}$. 

Let us write $b_0=e^t\delta^{1-3\eta}$. Also for every $w\in F_{\rm hm}$, let 
\begin{align*}
    I(w)=\{r\in [0, 1]: m^\delta(\pi_r(w))\geq C_0e^{-t/10}\delta^{\alpha-8\eta}\}.
\end{align*}
By the assumption we have $|I(w)| \geq C\delta^{\eta}$.

%

Recall that $e^t>\delta^{-\sqrt\eta}$ for the rest of the argument. 
Fix a maximal $b_0/2$ separated subset $\{w_1,\ldots, w_M\}$ of $F_{\rm hm}$.
Then $\{B(w_i, 4b_0)\}$ is a covering of $F_{\rm hm}$ with multiplicity $\ll 1$.  

 
Let $\nu_i$ denote the measure on $B(0,1)$ obtained as the image of 
\[
\tfrac{1}{\mu(B(w_i,4b_0))}\mu|_{B(w_i,4b_0)}
\] 
under the map $w \mapsto \frac{w_i-w}{4b_0}$. 


Then for any $1\leq i\leq M$ and all $r\geq \delta_0/4b_0$, we have 
\begin{align*}
\nu_i(B(z,r))\ll \tfrac{\mu (B(z', rb_0))}{\mu(B(w_i,4b_0))}\ll \tfrac{b_0^\alpha}{\mu(B(w_i,4b_0))} r^{\alpha}.
\end{align*}

Note that since $\delta \geq e^t \delta_0$, $r = e^{-2t} \geq \delta_0/4b_0$. For $\sigma=\nu_i$, apply Theorem~\eqref{thm: one dim optimal} with $\nu_i$, $\beta=1$, $r=e^{-2t}$, and 
$\vare_0=\sqrt\eta$. Then there exists a subset $E_i\subset B(w_i,b_0)$, with 
\begin{align*}
\nu_i(B(w_i,b_0)\setminus E_i)\leq e^{-\vare_0t}\leq \delta^{\eta}
\end{align*}
and for every $w\in E_i$, there is a subset $J_w$ with $|[0,1]\setminus J_w|\leq e^{-\vare_0t}\leq \delta^{\eta}$ so that if $r\in J_w$, then 
\be\label{eq: the set Ei}
\begin{aligned}
\nu_i(\{w': |(1,r,\tfrac{r^2}{2})\cdot \tfrac{w-w'}{b_0}|\leq e^{-2t}\})\ll  \tfrac{e^{\alpha t}\delta^{\alpha(1-3\eta)}}{\mu(B(w_i,4b_0))} e^{(-2+2\sqrt\eta)t}.
\end{aligned}
\ee
We emphasize here that since $\alpha \geq 1$, we can apply Theorem~(\ref{thm: one dim optimal}) with $\beta = 1$. 

For every $w\in F_{\rm hm}$, let $\hat I(w)=I(w)\cap J_w$. Then $|\hat I(w)|\geq \delta^{\eta}$ if $C > 2$. 
Moreover, for every $w\in F_{\rm hm}$ and every $r\in \hat I(w)$, 
\begin{multline*}
\{w'\in B(w,b_0): \|\pi_r(w)-\pi_r(w')\|\leq \delta\}\subset \\
\{w'\in B(w_i,4b_0): |(1,r,\tfrac{r^2}{2})\cdot \tfrac{w-w'}{b_0}|\leq e^{-2t}\}
\end{multline*}
where $w\in B(w_i,b_0)$ (see~\eqref{eq: def pi r dim 3} for the definition of $\pi_r=\pi_{t, r}$ and recall that $b_0=e^t\delta^{1-3\eta}$). 
Therefore, applying~\eqref{eq: the set Ei}, we obtain 
\be\label{eq: Using Ei}
\begin{aligned}
m^{\delta}(\pi_{r}(w)|B(w,b_0))
={}& \mu (\{w'\in B(w,b_0): \|\pi_r(w)-\pi_r(w')\|\leq \delta\})\\
\ll{}& e^{(-2+\alpha+2\sqrt\eta)t}\delta^{\alpha(1-3\eta)} \leq e^{-t/3}\delta^{\alpha-5\eta}.
\end{aligned}
\ee 

In view of~\eqref{eq: Using Ei} and pigeonhole principle, there is a subset $F''_{\rm hm}\subset F_{\rm hm}$ with $\mu(F''_{\rm hm})\gg C\delta^{\eta}$ so that the following holds. For every $w\in F''_{\rm hm}$ and all $r\in \hat I(w)$, there exists a dyadic number $b>b_0=e^t\delta^{1-3\eta}$ so that
\begin{align}\label{eq: b for the last case}
m^\delta(\pi_r(w)|\mathsf D_b(w))\gg \frac{1}{|\log b_0|}C_0e^{-t/10}\delta^{\alpha-8\eta} \geq C_0e^{-t/10}\delta^{\alpha-7\eta}.
\end{align}


Now applying pigeonhole principle again and recalling that $\delta \ll_\eta 1$ so that $\delta^{\frac{\eta}{10}} \leq |\log \delta|^{-1}$, there exists $b\geq e^t\delta^{1-3\eta}$ and 
$F'_{\rm hm}\subset F_{\rm hm}''$ with 
\begin{align*}
\mu(F'_{\rm hm})\gg C\delta^{11\eta/10}
\end{align*}
so that for all $w\in F'_{\rm hm}$, we have  
\begin{align*}
|\{r\in [0,1]: m^\delta(\pi_r(w)|\mathsf D_b(w))\geq C_0e^{-t/10}\delta^{\alpha-7\eta}\}|\geq \delta^{11\eta/10}.
\end{align*}
The proof is complete. 
\end{proof}

\begin{proof}[Proof of Lemma~\ref{lem: main 1st}]
We assume $C$ is large and $\delta$ is small enough depending only on $\eta$ so that $\delta^{\frac{\eta}{10}} \leq |\log \delta|^{-1}$ throughout the proof.  

In view of Lemma~\ref{lem: small radius}, see in particular~\eqref{eq: F' hm}, 
we will replace $F_{\rm hm}$ by $F_{\rm hm}'$ and assume that there is some $b\geq e^t\delta^{1-3\eta}$ so that~\eqref{eq: r for F'} 
holds for all $w\in F_{\rm hm}$.
 
For every $w\in F_{\rm hm}$, set  
\begin{align*}
I'(w)=\{r\in [0,1]: m^\delta(\pi_r(w)|\mathsf D_b(w))\geq C_0e^{-t/10}\delta^{\alpha-7\eta}\}.
\end{align*}
Then $|I'(w)|\gg \delta^{11\eta/10}$, see~\eqref{eq: r for F'}.  
Choose three subsets $I'_j(w)\subset I'(w)$ for $j=1,2,3$, which satisfy the following properties 
\be\label{eq: three intervals}
|I'_j(w)|\gg \delta^{2\eta}\qquad\text{and}\qquad {\rm dist}(I'_i(w), I'_j(w))\gg \delta^{11\eta/10} \quad\text{for $i\neq j$}.
\ee 
For $j=1,2,3$, define 
\begin{align*}
E_j(w)=\{w'\in F\cap\mathsf D_b(w):\|\pi_r(w)-\pi_r(w')\|\leq \delta\text{ for some $r\in I'_j(w)$}\}. 
\end{align*}
We claim
\be\label{eq: meas of Ej}
\mu(E_j(w))\geq C_0be^{-t/10}\delta^{\alpha-1-5\eta}\qquad\text{for $j=1,2,3$.}
\ee

Fix one $j$ and let $C'$ be a large (absolute) constant which will be explicated later. Take a maximal $\frac{C'\delta}{b}$-separated set $\{r_i\}_{i = 1}^N$ of $I_j'(w)$. We have
\begin{align*}
    N \cdot \frac{2C'\delta}{b} \geq |I_j'(w)| \gg \delta^{2\eta}
\end{align*}
and thus
\begin{align}\label{eq: bounding N}
    N \gg \delta^{2\eta}/(C'\delta/b)=b\delta^{2\eta}/(C'\delta)\gg b\delta^{2\eta-1}.
\end{align}
Since $b\geq e^t\delta^{1-3\eta}$,~\eqref{eq: bounding N} in particular implies that $N\gg e^t\delta^{-\eta}$. 

For every $1\leq i\leq N$, let 
\begin{align*}
    E_{j,i} ={}& \{w'\in E_j(w): \|\pi_{r_i}(w)-\pi_{r_i}(w')\|\leq \delta\}\\
    ={}& \{w' \in F \cap \mathsf{D}_b(w): \|\pi_{r_i}(w)-\pi_{r_i}(w')\|\leq \delta\}.
\end{align*}
Then $\mu(E_{j,i}) = m^{\delta}(\pi_{r_i}(w)|\mathsf{D}_b(w)) \geq C_0e^{- t/10}\delta^{\alpha-7\eta}$. 

Moreover, $E_{j,i}\cap E_{j,i'}=\emptyset$ for $i \neq i'$. 
Indeed if $w'\in E_{j,i}\cap E_{j,i'}$, then 
\begin{align*}
r_i,r_{i'}\in\{r\in[0, 1]: \|\pi_r(w)-\pi_{r}(w')\|\leq \delta\}.
\end{align*} 
This implies the set in the right hand side has diameter at least $\frac{C'\delta}{b}$. This leads to a contradiction. Indeed, write $w = (x, y, z)$ and $w' = (x', y', z')$. Recall that $\|w - w'\| \geq b$. If $\max\{|y - y'|, |z - z'|\} \leq \frac{b}{10}$, we have $|x - x'| \geq b$ and
\begin{align*}
    e^{-t} \delta \geq |(1, r, r^2/2) \cdot (w - w')| \geq b/2 \gg e^t \delta^{1 - 3\eta}.
\end{align*}
This leads to a contradiction. Now if $|z - z'| \geq \frac{b}{20}$, we have
\begin{align*}
    2\delta \geq |(0, 1, r_i) \cdot (w - w') - (0, 1, r_{i'}) \cdot (w - w')| = |r_i - r_{i'}|\cdot|z - z'| \geq C'\delta/20.
\end{align*}
This leads to a contradiction if $C'$ is large enough. Similarly, if $|z - z'| \leq \frac{b}{20}$ then $|y - y'| = \max\{|y - y'|, |z - z'|\} \geq \frac{b}{10}$. We have
\begin{align*}
    2e^{-t}\delta \geq{}& \left|(1, r_i, \frac{r_i^2}{2}) \cdot (w - w') - (1, r_{i'}, \frac{r_i^2}{2}) \cdot (w - w')\right|\\
    ={}& |r_i - r_{i'}| \cdot \left|(y - y') + \frac{1}{2}(r_i + r_{i'})(z - z')\right|\\
    \gg{}& C'\delta.
\end{align*}\
This leads to a contradiction if $C'$ is large enough. 

Using~\eqref{eq: bounding N} and $\mu(E_{j,i})\geq C_0e^{-t/10}\delta^{\alpha-7\eta}$, thus  
\begin{align*}
\mu(E_j(w))\gg N \cdot \mu(E_{j,i})\gg C_0be^{-t/10}\delta^{\alpha-1-5\eta},
\end{align*}
for $j=1,2,3$, as we claimed in~\eqref{eq: meas of Ej}.
 
Using $\mu(F_{\rm hm})\gg \delta^{3\eta/2}$ and~\eqref{eq: meas of Ej}, we conclude that  
\begin{multline}\label{eq:  lower bound}
\mu^4(\{(w,w_1, w_2, w_3)\in F_{\rm hm}\times F^3: w_j\in E_j(w)\})\gg \\C_0^3b^3e^{-\frac{3}{10}t}\delta^{3\alpha-3-13\eta}.
\end{multline}
 
We now find an upper bound for the measure of the set on the left side of~\eqref{eq:  lower bound}.  
To that end, fix $(w_1, w_2, w_3)\in F^3$ and let $\tilde{\Omega} = \tilde{\Omega}(w_1, w_2, w_3)$ be the subset of $F_{\mathrm{hm}}$ consisting of $w$ so that $w_j \in E_j(w)$. 
All $w \in \tilde{\Omega}$ satisfies that $w_j\in \mathsf D_b(w)$, where $b\geq e^t\delta^{1-3\eta}$. 
Moreover, there exists $r_j = r_j(w) \in I'_j(w)$ be so that $\|\pi_{r_j}(w)-\pi_{r_j}(w_j)\|\leq \delta$, that is, 
\be\label{eq: tangecy}
\begin{aligned}
|(1,r_j,r_j^2/{2})\cdot (w-w_j)|&\leq e^{-t}\delta\\
|(0,1,r_j)\cdot (w-w_j)|&\leq \delta
\end{aligned}
\ee
for all $j\in\{1,2,3\}$. From the construction of $I_j'(w)$, we have $|r_i - r_j| \gg \delta^{11\eta/10}$ for $i \neq j$. Let $\Omega$ be the set of $v \in \R^3$ so that
\begin{align*}
    b \leq \|v - w_j\| \leq 2b
\end{align*}
for $j = 1, 2, 3$ and there exists $r_j = r_j(v) \in [0, 1]$ for $j = 1, 2, 3$ satisfying
\begin{align}
\begin{aligned}
&|r_i - r_j| \gg \delta^{\frac{11\eta}{10}},\\
&|(1,r_j,r_j^2/{2})\cdot (v-w_j)| \leq e^{-t}\delta, \text{ and }\\
&|(0,1,r_j)\cdot (v-w_j)| \leq \delta.
\end{aligned}
\end{align}
The above discussion shows that $\tilde{\Omega} \subseteq \Omega$. 

By Lemma~\ref{lem: Quadratic formula} (a variant of Marstrand's three circles lemma) below, the set $\Omega$ is covered by $O(1)$-many sets with diameter $\ll e^{-t}\delta^{1 - \frac{11}{5}\eta}$.
Note that since $w_j \in \mathsf{D}_b(w)$, $\|w_1 - w_j\| \leq 4b$ for all $j$. This implies
\begin{align}\label{eqn: upper bound}
\begin{aligned}
&\mu^4(\{(w,w_1, w_2, w_3)\in F_{\rm hm}\times F^3: w_j\in E_j(w)\})\\
\leq{}& \int_F\int_{F^2 \cap B(w_1, 4b)^2}\mu(\tilde{\Omega}(w_1, w_2, w_3)) \,\mathrm{d}\mu^2(w_2, w_3) \,\mathrm{d}\mu(w_1)\\
\ll{}& C_0^3e^{-\alpha t}\delta^{\alpha- \frac{11}{5}\eta\alpha} b^{2\alpha}.
\end{aligned}
\end{align}

Comparing upper bound Eqn.~(\ref{eqn: upper bound}) and lower bound Eqn.~(\ref{eq:  lower bound}), we conclude  
\begin{align*}
b^3e^{-\frac{3}{10} t}\delta^{3\alpha-3-13\eta}\ll e^{-\alpha t}\delta^{\alpha- \frac{11}{5}\eta\alpha} b^{2\alpha},
\end{align*}
which implies 
\begin{align*}
\delta^{2\alpha-3-13\eta+\frac{11}{5}\eta\alpha}b^{3-2\alpha}\ll e^{(-\alpha +\frac3{10}) t}.
\end{align*}
Now using $b\geq e^t\delta^{1-3\eta}\geq \delta^{1-3\eta}$, $3 - 2\alpha \geq 0$, and the above, we conclude that 
\begin{align}\label{eq: final contr}
\delta^{2\alpha-3-13\eta+\frac{11}{5}\eta\alpha}\delta^{(3-2\alpha)(1-3\eta)}=\delta^{-13\eta+\frac{11}{5}\eta\alpha-(9-6\alpha)\eta} \ll e^{(-\alpha+\frac3{10}) t}.
\end{align}
However, $-13\eta+\frac{11}{5}\eta\alpha-(9-6\alpha)\eta \leq -6\eta$, since $3 - 2\alpha \geq 0$. 
Assuming $\delta$ is small enough and recalling that $3/10 < 1 \leq \alpha$,~\eqref{eq: final contr} cannot hold. 
This contradiction completes the proof.   
\end{proof}

\begin{lemma}\label{lem: Quadratic formula}
Fix $w_j \in B_{\R^3}(0, 1)$ for $j = 1, 2, 3$. Let $\Omega$ be the set of $v \in \R^3$ so that
\begin{align*}
    b \leq \|v - w_j\| \leq 2b
\end{align*}
for $j = 1, 2, 3$ and there exists $r_j = r_j(v) \in [0, 1]$ for $j = 1, 2, 3$ satisfying
\begin{align}\label{eq: tangecy'}
\begin{aligned}
&|r_i - r_j| \gg \delta^{\frac{11\eta}{10}},\\
&|(1,r_j,r_j^2/{2})\cdot (v-w_j)| \leq e^{-t}\delta, \text{ and }\\
&|(0,1,r_j)\cdot (v-w_j)| \leq \delta.
\end{aligned}
\end{align}

Then $\Omega$ is covered by $O(1)$-many sets of diameter $O(e^{-t}\delta^{1-\frac{11}{5}\eta})$. 
\end{lemma}

\begin{proof}
The argument is a variant of the proof of the Marstrand's three circles lemma, see \cite[Lemma 3.2]{Wolff-Survey} and \cite[Lemma 3.1]{KW99}. In fact, we will show that $\Omega$ can be covered by at most two sets with diameter $\ll e^{-t}\delta^{1 - \frac{11}{5}\eta}$.  

Write $w_j=(x_j,y_j,z_j)$. For all $v = (x, y, z) \in \Omega$, we write 
\begin{align*}
    \xi_j(v) = - \tfrac{y-y_j}{z-z_j}.
\end{align*}
We first claim that   
\begin{align}\label{eq: z-zj}
    |z-z_j|\geq \tfrac{b}{20}\qquad\text{and} \qquad \left|\xi_j(v) - r_j(v)\right|\leq 20e^{-t}\delta^{3\eta}
\end{align}
for $j\in\{1,2,3\}$.

To see these, first observe that if $\max\{|y - y_j|,|z - z_j|\}\leq b/10$, then using
$\|v - w_j\|\geq b\geq e^{t}\delta^{1-3\eta}$, we get $|x - x_j|\geq 9b/10$. Therefore,  
\begin{align*}
    \left|(v - w_j)\cdot(1, r_j(v), \tfrac{r_j(v)^2}{2})\right|\geq \tfrac{b}2\geq \tfrac{e^{t}\delta^{1-3\eta}}2,
\end{align*}
which contradicts the second equation in~\eqref{eq: tangecy'}. 

Hence, $\max\{|y - y_j|,|z - z_j|\}> b/10$. Suppose now contrary to the claim in~\eqref{eq: z-zj} that $|z - z_j|< b/20$. 
Then $|y - y_j|>b/10$, and we conclude 
\begin{align*}
|(y - y_j)+(z - z_j)r_j|=|(v -w_j)\cdot (0,1,r_j)|\geq \tfrac{b}{20}\geq \tfrac{e^{t}\delta^{1-3\eta}}{20}
\end{align*}
which contradicts the third equation in~\eqref{eq: tangecy'}.

Altogether, we have shown $|z -z_j|\geq b/20\geq {e^{t}\delta^{1-3\eta}}/{20}$. 
Now using the third equation in~\eqref{eq: tangecy'} again, we conclude that
\be\label{eq: y-yj/z-zj}
\left|\xi_j(v) - r_j(v)\right| = \left|\tfrac{y - y_j}{z - z_j}+r_j(v)\right|\leq \frac{20\delta}{b}\leq 20e^{-t}\delta^{3\eta}.
\ee

We now argue as in the proof of Marstrand's three circles lemma. We will utilize the inverse function theorem to reduce the problem to an algebraic problem on tangency of standard parabolas, see Lemma~\ref{lem:tangency for std parabola} below. 

First note that by completing squares, we get
\begin{align*}
(1,r, \tfrac{r^2}2)\cdot (v-w_j)= \tfrac{|(0, 1, r)\cdot(v-w_j)|^2}{2(z - z_j)}+(x - x_j)-\tfrac{(y - y_j)^2}{2(z - z_j)}.
\end{align*}
Evaluate this equation at $r_j(v)$ for $j\in\{1,2,3\}$.
Then \eqref{eq: tangecy'} and $|z - z_j|\geq b/20\geq e^t\delta^{1-3\eta}/20$, see~\eqref{eq: z-zj}, imply that 
\begin{align}\label{eq: quadratic formula}
\left|(x - x_j)-\tfrac{(y - y_j)^2}{2(z - z_j)}\right| \leq e^{-t}\delta + 10e^{-t}\delta^{1 + 3\eta} \leq 20 e^{-t}\delta \quad \text{for $j\in\{1,2,3\}$.}
\end{align}
for $j\in\{1,2,3\}$.

Define $G:\Omega \to \R^3$ by  
\begin{align*}
    G(x,y,z)= \left((x - x_1) -\tfrac{(y-y_1)^2}{2(z-z_1)}, (x - x_2) - \tfrac{(y-y_2)^2}{2(z-z_2)}, (x - x_3) - \tfrac{(y-y_3)^2}{2(z-z_3)}\right).
\end{align*}
In view of~\eqref{eq: quadratic formula}, we have that 
\begin{align}\label{eq: tangecy and G}
    G(\Omega) \subseteq B(0, 20e^{-t} \delta).
\end{align}

We now apply the inverse function theorem to show that $0 \in G(\Omega)$. Note that for $v = (x, y, z)$, 
\begin{align}\label{eqn:jacobian}
{\rm D}G(v)=
\begin{pmatrix}
1 & -\tfrac{y-y_1}{z-z_1} & \tfrac{1}{2}\Bigl(\tfrac{y-y_1}{z-z_1}\Bigr)^2\\
1 & -\tfrac{y-y_2}{z-z_2} & \tfrac{1}{2}\Bigl(\tfrac{y-y_2}{z-z_2}\Bigr)^2\\
1 & -\tfrac{y-y_3}{z-z_3} & \tfrac{1}{2}\Bigl(\tfrac{y-y_3}{z-z_3}\Bigr)^2
\end{pmatrix} = 
\begin{pmatrix}
1 & \xi_1(v) & \frac{1}{2}\xi_1(v)^2\\
1 & \xi_2(v) & \frac{1}{2}\xi_2(v)^2\\
1 & \xi_3(v) & \frac{1}{2}\xi_3(v)^2
\end{pmatrix}.
\end{align}
The standard computation for Vandermonde determinant shows that 
\begin{align*}
    |\det \mathrm{D}G(v)| \asymp |\xi_1(v) - \xi_2(v)|\cdot|\xi_2(v) - \xi_3(v)|\cdot|\xi_3(v) - \xi_1(v)|.
\end{align*}

Recall now that $|r_i(v) - r_j(v)| \gg \delta^{11\eta/10}$ for $i \neq j$, and recall from~\eqref{eq: z-zj} that 
\begin{align*}
    \left|\xi_j(v) - r_j(v)\right| = \left|\tfrac{y_0-y_j}{z_0-z_j}+r_j\right| \leq 20e^{-t}\delta^{3\eta}.
\end{align*}
Hence 
\begin{align}\label{eqn: tangency are isolated for different parabola}
    \left|\xi_j(v) - \xi_i(v)\right| \gg \delta^{11\eta/10}
\end{align}
for $i \neq j$ if $\delta$ is small enough depending only on $\eta$. We conclude 
\begin{align}\label{eqn: det estimate}
\begin{aligned}
    &\|{\rm D}G(v)\| \ll 1,\\
    &\delta^{33\eta/10}\ll |\det {\rm D}G(v)|\ll 1, \quad \text{ for all } v \in \Omega.
\end{aligned}
\end{align}
The standard computation for inverse of Vandermonde matrix shows that
\begin{align}\label{eqn:inverse of DG}
    \|{\rm D}G(v)^{-1}\| \ll \frac{\max_{i \neq j}|\xi_i(v) - \xi_j(v)|}{|\det \mathrm{D}G(v)|} \ll \delta^{-\frac{11\eta}{5}}.
\end{align}

Fix $w = (x_0, y_0, z_0) \in \Omega$ and let $A > 1$ be a large (absolute) constant which will be explicated later. Define
\begin{align*}
    E_{A, t, \delta}(w) = \{v \in \R^3: \|{\rm D}G(w)(w - v)\| \leq Ae^{-t}\delta\}.
\end{align*}
Eqn.~(\ref{eqn: det estimate}) and (\ref{eqn:inverse of DG}) imply that
\begin{align}\label{eqn:ellipsoid diam and vol}
\begin{aligned}
    &\diam E_{A, t, \delta}(w) \ll A e^{-t}\delta^{1 - \frac{11\eta}{5}},\\
    &|E_{A, t, \delta}(w)| \ll A^3e^{-3t}\delta^{3 - \frac{33}{10}\eta}.
\end{aligned}
\end{align}

We now show that for all $v = (x, y, z) \in E_{A, t, \delta}(w)$ and $j = 1, 2, 3$, 
\begin{align}\label{eqn:tangencyare close}
    |\xi_{j}(v) - \xi_j(w)| = \left|\frac{y - y_j}{z - z_j} - \frac{y_0 - y_j}{z_0 - z_j}\right| \ll Ae^{-2t} \delta^{\frac{19\eta}{10}}.
\end{align}
A direct computation shows that 
\begin{align*}
    \frac{y - y_j}{z - z_j} - \frac{y_0 - y_j}{z_0 - z_j} ={}& \frac{1}{z - z_j}\left((y - y_0) - \frac{y_0 - y_j}{z_0 - z_j}(z - z_0)\right)\\
    ={}& \frac{1}{z - z_j}\left((y - y_0) + \xi_i(w)(z - z_0)\right).
\end{align*}
Note that by \eqref{eq: z-zj} we have
\begin{align*}
    |z - z_j| \geq |z_0 - z_j| - |z - z_0| \geq \frac{b}{20} - Ae^{-t}\delta^{1 - \frac{11\eta}{5}} \geq \frac{b}{100}.
\end{align*}
For the numerator, recall from the definition of $E_{A, t, \delta}(w)$ that $\|\mathrm{D}G(w)(v - w)\| \leq Ae^{-t}\delta$.  Considering the difference between $i$-th row and $j$-th row using Eqn.~(\ref{eqn:jacobian}), we have
\begin{align}\label{eqn: Property of E}
    \left|\xi_i(w) - \xi_j(w)\right|\cdot\left|(y - y_0) -\frac{1}{2}\left(\xi_i(w) + \xi_j(w)\right)(z - z_0)\right| \leq 2A e^{-t} \delta.
\end{align}
Recall that $|\xi_i(w) - \xi_j(w)| \gg \delta^{\frac{11\eta}{10}}$ if $i \neq j$. We have for all $i \neq j$
\begin{align*}
    \left|(y - y_0) + \frac{1}{2}\left(\xi_i(w) + \xi_j(w)\right)(z - z_0)\right| \ll A e^{-t} \delta^{1 - \frac{11\eta}{10}}.
\end{align*}
Let $\{i, j, k\} = \{1, 2, 3\}$. Note that we have
\begin{align}\label{eqn: xi}
    \xi_j(w) = \frac{1}{2}(\xi_j(w) + \xi_i(w)) + \frac{1}{2}(\xi_j(w) + \xi_k(w)) - \frac{1}{2}(\xi_i(w) + \xi_k(w)). 
\end{align}
Combining \eqref{eqn: Property of E} and \eqref{eqn: xi}, we have
\begin{align*}
     \left|\frac{y - y_j}{z - z_j} - \frac{y_0 - y_j}{z_0 - z_j}\right| ={}& \frac{1}{|z - z_j|}\left|(y - y_0) + \xi_j(w)(z - z_0)\right|\\
     \ll{}& \frac{A e^{-t} \delta^{1 - \frac{11\eta}{10}}}{b} = A e^{-2t}\delta^{\frac{19\eta}{10}}.
\end{align*}
This proves Eqn.~(\ref{eqn:tangencyare close}). 

We now claim that for all $v \in E_{A, t, \delta}(w)$, $\|{\rm D}G(v) {\rm D}G(w)^{-1} - I_3\| \leq \frac{1}{100}$. In fact, we will show that $\|{\rm D}G(v) {\rm D}G(w)^{-1} - I_3\| \ll A\delta^{\sqrt{\eta}}$. The claim follows if $\delta$ is small enough depending on $\eta$. 

For simplicity, we write $\xi_j = \xi_j(w)$. 
Note that
\begin{align*}
    \|\mathrm{D}G(v) - \mathrm{D}G(w)\| \ll \max_{j}|\xi_j(v) - \xi_j(w)| \ll A e^{-2t}\delta^{\frac{19\eta}{10}}.
\end{align*}
We have
\begin{align*}
     \|(\mathrm{D}G(v) - \mathrm{D}G(w))\mathrm{D}G(w)^{-1}\| \ll{}& \|\mathrm{D}G(w)^{-1}\|\max_{j}|\xi_j(v) - \xi_j(w)|\\
     \ll{}& A e^{-2t}\delta^{-\frac{3\eta}{10}} \leq A \delta^{2\sqrt{\eta} - \frac{3}{10}\eta} \leq A\delta^{\sqrt{\eta}}.
\end{align*}
The last two inequalities follows from the assumption that \emph{$t$ is large}, i.e., $e^t \geq \delta^{-\sqrt{\eta}}$. 

The proof of inverse function theorem using Newton's method implies that there exists $(\hat{x}, \hat{y}, \hat{z}) \in E_{A, t, \delta}(w)$ so that $G(\hat{x}, \hat{y}, \hat{z}) = 0$ if $A$ is large enough depending only on an absolute constant. Let $\hat{\gamma}(r) = \hat{x} + \hat{y}r + \hat{z}\frac{r^2}{2}$ and $\gamma_i(r) = x_i + y_ir + z_i\frac{r^2}{2}$ for $i = 1, 2, 3$ be the corresponding (possibly degenerate) parabolas in standard forms. Note that $G(\hat{x}, \hat{y}, \hat{z}) = 0$ is equivalent to $\hat{\gamma}$ being tangent to $\gamma_j$ at $\left(\frac{\hat{y} - y_j}{\hat{z} - z_j}, \hat{\gamma}\left(\frac{\hat{y} - y_j}{\hat{z} - z_j}\right)\right)$ for $j = 1, 2, 3$. By \eqref{eqn:tangencyare close} and \eqref{eqn: tangency are isolated for different parabola}, we have
\begin{align*}
    \left|\frac{\hat{y} - y_j}{\hat{z} - z_j} - \frac{\hat{y} - y_i}{\hat{z} - z_i}\right| \gg \delta^{\frac{11}{10}\eta}.
\end{align*}
Therefore $\hat{\gamma}$ is tangent to $\gamma_j$'s at three distinct points. Lemma~\ref{lem:tangency for std parabola} implies there are at most two choices of the parabola $\hat{\gamma}$ and hence $(\hat{x}, \hat{y}, \hat{z})$. Combining with Eqn.~(\ref{eqn:ellipsoid diam and vol}), the proof of the lemma is complete. 
\end{proof}

\begin{remark}
The proof of Lemma~\ref{lem: Quadratic formula} is similar to the proof of Marstrand's three circles lemma except one slight simplification due to the \emph{$t$ large} scenario. 

In the estimate of $\|\mathrm{D}G(v)\mathrm{D}G(w)^{-1} - \mathrm{I}_3\|$, the trivial estimate
\begin{align*}
    \|(\mathrm{D}G(v) - \mathrm{D}G(w))\mathrm{D}G(w)^{-1}\| \leq{}& \|\mathrm{D}G(v) - \mathrm{D}G(w)\|\|\mathrm{D}G(w)^{-1}\|\\
    \ll{}& Ae^{-2t} \delta^{-\frac{3\eta}{10}} \leq \delta^{\sqrt{\eta}}
\end{align*}
already provides useful bound to apply the inverse function theorem. 

We sketch a proof of the lemma in both \emph{$t$ large} and  \emph{$t$ small} scenarios. For simplicity, we write $\xi_j = \xi_j(w)$. 
We have
\begin{align*}
\mathrm{D}G(w)^{-1}
= \frac{1}{\det \mathrm{D}G(w)}\begin{pmatrix}
    \frac{1}{2}(\xi_2\xi_3^2 - \xi_3 \xi_2^2) & \frac{1}{2}(\xi_1\xi_3^2 - \xi_3\xi_1^2) & \frac{1}{2}(\xi_1\xi_2^2 - \xi_2 \xi_1^2)\\
    \frac{1}{2}(\xi_2^2 - \xi_3^2) & \frac{1}{2}(\xi_1^2 - \xi_3^2) & \frac{1}{2}(\xi_1^2 - \xi_2^2)\\
    \xi_3 - \xi_2 & \xi_3 - \xi_1 & \xi_2 - \xi_1
\end{pmatrix}.
\end{align*}
Direct computation implies that entries of $(\mathrm{D}G(v) - \mathrm{D}G(w))\mathrm{D}G(w)^{-1}$ are of the form
\begin{align*}
    (\det \mathrm{D}G(w))^{-1} (\xi_i(v) - \xi_i(w))(\xi_j(w) + \xi_k(w) - \xi_i(v) - \xi_i(w))(\xi_j(w) - \xi_k(w))
\end{align*}
where $j \neq k$. 

Note that by Eqn.~(\ref{eqn:tangencyare close}) and Eqn~(\ref{eqn: tangency are isolated for different parabola}), we have
\begin{align}\label{eqn: tangency are isolated for different parabola under perturbation}
    |\xi_i(v) - \xi_j(w)| \leq 2|\xi_i(w) - \xi_j(w)|
\end{align}
for $i \neq j$ if $\delta \ll_{\eta} 1$. 
The estimates Eqn.~(\ref{eqn:tangencyare close}) and Eqn.~(\ref{eqn: tangency are isolated for different parabola under perturbation}) show that 
\begin{align*}
    |\xi_j(w) + \xi_k(w) - \xi_i(v) - \xi_i(w)| \leq 4\max_{j \neq k} |\xi_j(w)  - \xi_k(w)|.
\end{align*}
Combining the above estimate with the computation of the determinant, we have
\begin{align*}
     \|(\mathrm{D}G(v) - \mathrm{D}G(w))\mathrm{D}G(w)^{-1}\|
     \ll \delta^{-\frac{11\eta}{10}}\max_{i}|\xi_i(v) - \xi_i(w)| \ll{}& A e^{-2t} \delta^{\frac{4}{5}\eta}\\
     \leq{}& A\delta^{\frac{4}{5}\eta}.
\end{align*}
This implies $\|(\mathrm{D}G(v) - \mathrm{D}G(w))\mathrm{D}G(w)^{-1}\| \leq \frac{1}{100}$ if $\delta \ll_{\eta} 1$. 
\end{remark}

\begin{lemma}\label{lem:tangency for std parabola}
    Let $\gamma_i(r) = (r, x_i + y_ir + z_i\frac{r^2}{2})$ for $i = 1, 2, 3$ be three (possibly degenerate) parabolas in standard form. Then there are at most two choices of $(\hat{x}, \hat{y}, \hat{z}) \in \R^3$ so that the parabola $\hat{\gamma}(r) = (r, \hat{x} + \hat{y}r + \hat{z}\frac{r^2}{2})$ is tangent to $\gamma_i$'s at three distinct points. 
\end{lemma}

\begin{proof}
    The lemma is an analogue of Appolonius' theorem and it  follows from \cite[Lemma 3.5]{KW99} as the function $\Phi$ constructed in Eqn.~(\ref{eqn:Phi circle}) satisfies the cinematic curvature condition in \cite{KW99}. We provide an elementary algebraic proof of the current case for the convenience of the reader. 

    If some $\gamma_i$'s are the same, then there exists $\gamma_{i_0}$ so that $\hat{\gamma}$ is tangent to $\gamma_{i_0}$ at two distinct points. This implies that $\hat{\gamma} = \gamma_{i_0}$ and proves the lemma. From now on we assume that $\gamma_i$'s are distinct. 

    Note that $\hat{\gamma}$ being tangent to all $\gamma_i$'s is equivalent to $G(\hat{x}, \hat{y}, \hat{z}) = 0$. Without loss of generality, we can assume $(x_1, y_1, z_1) = (0, 0, 0)$. We have
    \begin{align*}
    \begin{cases}
        &2\hat{x}\hat{z} = \hat{y}^2, \\
        &2(\hat{x} - x_2)(\hat{z} - z_2) = (\hat{y} - y_2)^2, \\
        &2(\hat{x} - x_3)(\hat{z} - z_3) = (\hat{y} - y_3)^2.
    \end{cases}
    \end{align*}
    This system of equation is equivalent to 
    \begin{align*}
    \begin{cases}
        &2\hat{x}\hat{z} = \hat{y}^2, \\
        &2z_2\hat{x} - 2y_2\hat{y} + 2x_2 \hat{z} =  2x_2z_2 - y_2^2, \\
        &2z_3\hat{x} - 2y_3\hat{y} + 2x_3 \hat{z} =   2x_3z_3 - y_3^2.
    \end{cases}
    \end{align*}

    Since $\gamma_i$'s are distinct parabolas, $(x_i, y_i, z_i) \neq (0, 0, 0)$ for $i = 2, 3$. Therefore the latter two equations define planes $P_2$ and $P_3$ in $\R^3$. Moreover, we claim that $P_2 \neq P_3$ and at least one of them does not pass through the origin. If both planes contain the origin, then the discriminants satisfy $2x_2z_2 - y_2^2 = 2x_3z_3 - y_3^2 = 0$. This implies that both $\gamma_2$ and $\gamma_3$ are tangent to $\gamma_1$, the horizontal line. Therefore $\hat{\gamma}$ is tangent to $\gamma_i$'s for all $i$ at exactly one point, which contradict to the assumption of the lemma. Now we suppose $P_2 = P_3$, and without loss of generality $(x_2, y_2, z_2) = \lambda (x_3, y_3, z_3)$ for some $\lambda \in \R$ and $2x_2z_2 - y_2^2 = \lambda(2x_3z_3 - y_3^2)$. Since one of the discriminants is nonzero, $\lambda = 0, 1$, contradicting to the assumption that $\gamma_i$'s are distinct. 

    Let $\mathcal{C}$ be the cone defined by $2\hat{x}\hat{z} = \hat{y}^2$. We assume without loss of generality that $P_2$ does not pass the origin. Then the intersection of $P_2$ with $\mathcal{C}$ is a non-degenerate quadratic curve, i.e., not union of two lines. Since $P_3 \neq P_2$, the solutions of the equations, if they exist, are contained in the intersection of a line and a non-degenerate quadratic curve. The latter consists of at most two points. The proof of the lemma is complete. 
\end{proof}


\section{Proof of Theorem~\ref{thm:Improving Dimension SL2 irrrep}}
We will prove Theorem~\ref{thm:Improving Dimension SL2 irrrep} in this section. Before starting the proof let us fix some notation. 

For all $r\in[0,1]$, we put
$\xi(r)=(\tfrac{r}{1!}, \tfrac{r^2}{2!},\ldots, \tfrac{r^{n}}{n!}, \tfrac{r^{n + 1}}{(n + 1)!}) \in \R^{n + 1}$.
For all $1 \leq k\leq n + 1$, let 
\begin{align*}
\mathfrak p_{r}^{(k)}:\R^{n + 1}\to \R^k
\end{align*}
be the projection onto the space spanned by 
$\{\xi'(r),\ldots, \xi^{(k)}(r)\}$ defined by
\begin{align*}
\mathfrak p_{r}^{(k)}(w)=\Bigl(w\cdot \xi'(r), w\cdot \xi^{(2)}(r),\ldots, w\cdot \xi^{(k)}(r)\Bigr)
\end{align*}
where $w\cdot v$ is the usual inner product on $\R^{n + 1}$.
We define $\mathfrak{p}^{(0)}_r$ to be the zero map for all $r \in [0, 1]$. 

Recall that we let
\begin{align*}
    u_r = \begin{pmatrix}
        1 & r\\
         & 1
    \end{pmatrix}
\end{align*}
and 
\begin{align*}
    a_{t} = \begin{pmatrix}
        e^{{t}/2} & \\
         & e^{-{t}/2}
    \end{pmatrix}.
\end{align*}
With this normalization, there exists a basis for the irreducible representation $V$ so that we can identify $V \cong \R^{n + 1}$ and for all $w \in \R^{n + 1}$, 
\begin{align*}
    u_r.w ={}& \biggl(w \cdot \xi^{(1)}(r), \cdots, w \cdot \xi^{(n + 1)}(r)\biggr),\\
    a_{{t}}. w = {}& (e^{\frac{n}{2}{t}}w_1, e^{\frac{n - 2}{2}{t}}w_2\cdots, e^{-\frac{n}{2}{t}}w_{n + 1}).
\end{align*}

For a probability measure $\mu$ supported on $B_{V}(0, 1)$, we define the following notions. If $\mu(A) > 0$, we let 
\begin{align*}
\mu_{A} = \tfrac{1}{\mu(A)} \mu|_{A}
\end{align*}
be the normalized restriction of $\mu$ on $A$. 

With this identification, the family of projections $\mathfrak{p}_r^{(k)}$ has the following two properties. 
\begin{enumerate}
    \item The projection $\mathfrak{p}^{(n + 1)}_r$ is exactly the action of $u_r$ in this representation. 
    \item The difference between $\mathfrak p_{r}^{(k)}$ and the orthogonal projection to the space spanned by $\{\xi'(r),\ldots, \xi^{(k)}(r)\}$ is a bi-Lipschitz map where the Lipschitz constant depends only on $n$. 
\end{enumerate}

In this section, we call cubes of the form $\prod_{i = 1}^{n + 1}[\frac{n_i}{2^k}, \frac{n_i + 1}{2^k})$ dyadic cubes. For a dyadic cube $Q \subset \R^{n + 1}$, we set $\Hom_Q$ to be the homothety that maps $Q$ to $[0, 1)^{n+1}$. We set 
\begin{align*}
\mu^Q = (\Hom_Q)_{\ast}\mu_Q
\end{align*}
to be the rescaled normalized restriction of $\mu$  on $Q$. We also refer to $\mu^Q$ as the conditional measure of $\mu$ on $Q$. We use the notion $\mathcal{D}_{\rho}$ for the collection of dyadic $\rho$-cubes and $A_Q = A \cap Q$. 

As it was already mentioned, we will prove Theorem~\ref{thm:Improving Dimension SL2 irrrep} using~\cite[Thm.~2.1]{GGW}. We record the following consequence of ~\cite[Thm.~2.1]{GGW}. 

\begin{thm}[\cite{GGW}]
\label{thm: proj GGW}
Let $1\leq k\leq n + 1$ and let $0<\alpha\leq k$.
Let $\mu$ be the uniform measure on a finite set $F\subset B_{\R^{n + 1}}(0,1)$ satisfying 
\begin{align*}
\mu(B_{\R^{n + 1}}(w, \delta))\leq C_0 \delta^\alpha\qquad\text{for all $w$ and all $\delta\geq \delta_0$}
\end{align*}
where $C_0>0$.

Let $0<\vare<10^{-4}\alpha$.
For every $\delta\geq \delta_0$, there exists a subset 
$J_{\delta}\subset [0,1]$ with $|[0,1]\setminus J_{\delta}|\ll_{\alpha, \vare}\delta^{\star\vare^2}$ so that the following holds. 
Let $r\in J_\delta$, then there exists a subset $F_{\delta,r}\subset F$ with 
\begin{align*}
\mu(F\setminus F_{\delta,r})\ll_{\alpha, \vare} \delta^{\star\vare^2}
\end{align*}
such that for all $w\in F_{\delta, r}$ we have 
\begin{align*}
\mu\Bigl(\{w'\in F: \|\mathfrak p_{r}^{(k)}(w)-\mathfrak p_{r}^{(k)}(w')\|\leq \delta\}\Bigr)\ll_\vare C_0\delta^{\alpha-\vare}.
\end{align*} 
\end{thm}

\begin{proof}
We deduce this from~\cite[Thm.~2.1]{GGW}. The argument is more or less standard. 
Indeed it is similar to the deduction of Theorem~\ref{thm: main finitary} from Lemma~\ref{lem: main 1st} 
and to (a finitary version of) the argument in~\cite[\S2]{GGW}. See also \cite[Section 2]{JL24}. 


Since $1\leq k\leq n + 1$ is fixed throughout the argument, we denote $\mathfrak p_{r}^{(k)}$ by $\mathfrak p_r$ in this proof. 
We adapt the notation $m^\delta$ from the previous section: for every $r \in [0, 1]$ and all $w\in F$, put 
\begin{align*}
m^\delta(\pfrak_r(w))=\mu\{w': \|\pfrak_r(w)-\pfrak_r(w')\|\leq \delta\}.
\end{align*}

Let $D_1$ and $D_2$ be two large constants which will be explicated later. 
Let $\eta=\vare/D_1$. For all $r \in [0, 1]$, let 
\begin{align*}
F_{\rm bad}(r)=\{w: m^\delta(\pfrak_r(w))\geq C_0\delta^{\alpha-D_1\eta}\}.
\end{align*}

Assume contrary to the claim in Theorem~\ref{thm: proj GGW}
that there exists a subset $I_{\rm bad}\subset [0, 1]$ with 
$|I_{\rm bad}|\geq D_2\delta^{\eta^2/2}$ so that for all $r\in I_{\rm bad}$, we have 
\begin{align*}
\mu(F_{\rm bad}(r))\geq D_2\delta^{\eta^2/2}.
\end{align*}
We will get a contradiction with~\cite[Thm.~2.1]{GGW}, provided that $D_i$'s are large enough. 

First note that, for every $r\in I_{\rm bad}$, the number of $\delta$-boxes $\{\mathsf B_{i,r}\}$ 
required to cover $\pfrak_r(F_{\rm bad}(r))$ is $\leq D_3C_0^{-1}\delta^{-\alpha+D_1\eta}$. 
Following~\cite{GGW}, let $\mathcal T_r=\{\mathbb T_{i,r}\}$ where $\mathbb T_{i,r}=\pfrak_r^{-1}(\mathsf B_{i,r})\cap B_{\R^{n + 1}}(0,1)$; 
note that 
\be\label{eq: number of Tr}
\#\mathcal T_r\leq D_3 C_0^{-1}\delta^{-\alpha+D_1\eta}.
\ee
 
Select a maximal $\delta$-separated subset $\Lambda_\delta\subset I_{\rm bad}$ and extend this to a maximal 
$\delta$-separated subset $\hat\Lambda_\delta$ of $[0,1]$. 

Let $\rho$ denote the uniform measure on $\hat\Lambda_\delta$, and as it was done in the proof of Theorem~\ref{thm: main finitary},
equip $\hat\Lambda_\delta\times F$ with the product measure $\rho\times \mu$, and let  
\begin{align*}
\mathsf E&=\{(r,w)\in \Lambda_\delta\times F: m^\delta(\pfrak_r(w))\geq C_0\delta^{\alpha-D_1\eta}\}\\
&=\{(r,w)\in \Lambda_\delta\times F: w\in F_{\rm bad}(r)\}.
\end{align*}
Then the above implies that $\rho\times \mu(\mathsf E)\geq D_2^2\delta^{\eta^2}$.

For every $w\in F$, let $\mathsf E_w=\{r\in\Lambda_\delta: (r,w)\in\mathsf E\}$, and   
set 
\begin{align*}
F'=\{w\in F: \rho(\mathsf E_w)\geq D_2 \delta^{\eta^2}\}.
\end{align*}
Then using Fubini's theorem, we conclude $\mu(F')\geq \tfrac12D_2^2\delta^{\eta^2}$.

Recall that $\rho$ is the normalized counting measure on $\Lambda_\delta$. The above definitions thus imply
\begin{align*}
\sum_{r\in\Lambda_\delta}1_{\mathbb T_r}(w)\geq D_3\delta^{\eta^2-1}\qquad\text{for all $w\in F'$}
\end{align*}
where $D_3=O(D_2)$ and the implied constant is absolute. 
Let $\mu'$ denote the restriction of $\mu$ to $F'$.
Now by~\cite[Thm 2.1]{GGW}, applied with $\eta^2$, $\mu'$
and $\{\mathcal T_r: r\in\Lambda_\delta\}$, we have  
\begin{align}
\begin{aligned}
\sum_{r\in\Lambda_\delta} \#\mathcal T_r&\geq D_4(n, \vare, \alpha) C_0^{-1}\mu'(\R^{n + 1})\delta^{-1-\alpha+D\eta}\\
&\geq \tfrac12D_2^2 D_4(n, \vare, \alpha) C_0^{-1} \delta^{\eta^2}\delta^{-1-\alpha+D\eta}
\end{aligned}
\end{align}
where $D=10^{10(n + 1)}$ and in the second line we used $\mu'(\R^{n + 1})=\mu(F')\geq \tfrac12D_2^2\delta^{\eta^2}$.
Thus there exists some $r\in\Lambda_\delta$ so that 
\be\label{eq: GGW and pigeonhole}
\#\mathcal T_r \geq \tfrac12C_0^{-1}D_4(n, \vare, \alpha) D_2^2 \delta^{-\alpha+(D+1)\eta}.
\ee
Now comparing~\eqref{eq: GGW and pigeonhole} and~\eqref{eq: number of Tr} and recalling $D_3 = O(D_2)$ we get a contradiction so long as $D_1 > D + 1$, $D_2$ large enough depending only on $\epsilon, n$ and $\alpha$. The proof is complete. 
\end{proof}

Now we prove Theorem~\ref{thm:Improving Dimension SL2 irrrep}. 

\begin{proof}
    Without loss of generality, we assume both $e^{\frac{n{t}}{2}}\delta$ and $\delta$ are dyadic numbers. This will only result in a constant factor depending only on the ambient space in the final estimate. 
    
    For simplicity, let $\rho = e^{\frac{n{t}}{2}}\delta$ and $k = \lfloor\alpha\rfloor$. Fix $\delta$ satisfying $e^{\frac{n}{2}{t}}\delta_0 \leq \delta \leq e^{-\frac{n}{2}{t}}$. 

    Let $P$ be the tube of size
    \begin{align*}
        e^{-n{t}} \times  \cdots \times e^{-{t}} \times 1
    \end{align*}
    centered at the origin, with its sides aligned along the coordinate axes, via the identification $V \cong \R^{n + 1}$. We note that $u_r^{-1} a_{{t}}^{-1} B(x, \delta)$ is a translate of $u_r^{-1}.\rho P$, where $\rho P$ is the image of $P$ scaled by a factor $\rho$. We denote it by $w + u_r^{-1}.\rho P$. 
    
    We will prove the theorem by estimating $\mu(w + u_r^{-1}.\rho P)$. Note that since the tube $w + u_r^{-1}.\rho P$ lies in $O(1)$ many dyadic $\rho$-cubes, it suffices to estimate $\mu_Q(w + u_r^{-1}.\rho P)$ for those dyadic $\rho$-cubes $Q$ intersecting $w + u_r^{-1}.\rho P$ non-trivially. Apply the homothety that map $Q$ to the cube $[0, 1)^{n + 1}$, we have
    \begin{align*}
        \mu_Q(w + u_r^{-1}.\rho P) = \mu^Q(w' + u_r^{-1}.P)
    \end{align*}
    where $w' = \Hom_Q (w)$. Therefore, it suffices to estimate $\mu^Q(w' + u_r^{-1}.P)$. To this end, we first study the dimension condition for those conditional measures $\mu^Q$. 
    
    For all $s \geq \rho^{-1}\delta_0$, we have
    \begin{align*}
        \mu^Q(B(x, s)) ={}& \mu_Q(\Hom_Q^{-1}B(x, s))\\
        ={}& \frac{1}{\mu(Q)}\mu|_Q(B(x', \rho s))\\
        \leq{}& \frac{C\rho^\alpha}{\mu(Q)}s^\alpha.
    \end{align*}

    Now we study the relation between $\mu^Q(w' + u_r^{-1}.P)$ with the projections 
    $\Bigl\{\mathfrak{p}_r^{(l)}\Bigr\}_{r, l}$. This follows from the standard `projection and slicing' picture, as we now explicate. Namely, in our case, for all $l = 0, \cdots, n + 1$, there exists some rectangle $R_{l}(w') \subset \R^l$ whose sides are parallel to the axes and of lengths
    \begin{align*}
        e^{-n{t}} \times \cdots \times e^{-(n - l + 1){t}}
    \end{align*}
    so that 
    \begin{align*}
        w' + u_r^{-1}.P \subseteq  \Bigl(\mathfrak{p}_r^{(l)}\Bigr)^{-1}(R_{l}(w')).
    \end{align*}
    It suffices to study $\Bigl(\mathfrak{p}_r^{(l)}\Bigr)_{\ast} \mu^Q(R_{l}(w'))$. We will focus on the cases where $l = k$ and $l = k + 1$. 

    Apply Theorem~\ref{thm: proj GGW} with the measures $\{\mu^Q\}_Q$, $\mathfrak{p}^{(k)}_r$, scale $e^{-n{t}}$, and $\epsilon/n$. Using Fubini's theorem, there exists $\mathcal{E}_{{t}, \delta} \subset [0, 1]$ with $|\mathcal{E}_{{t}, \delta}| \ll_{\alpha, \epsilon} e^{-\star\epsilon^2{t}}$ so that for all $r \in [0, 1] \setminus \mathcal{E}_{{t}, \delta}$, there exists $\mathcal{D}_{\rho}(r)$ with 
    \begin{align*}
    \sum_{Q \notin \mathcal{D}_{\rho}(r)} \mu(Q) \ll_{\alpha,\epsilon} e^{-\star\epsilon^2{t}}.
    \end{align*}
    Moreover, for all $Q \in \mathcal{D}_{\rho}(r)$, there exists $F_{Q}(r) \subset F_Q$ with 
    \begin{align*}
        \mu_Q(F_Q \setminus F_{Q}(r)) \ll_{\alpha, \epsilon} e^{-\star\epsilon^2{t}}
    \end{align*}
    so that for all $w \in V$, we have
    \begin{align*}
        \mu^Q \Bigl(\Hom_Q(F_Q(r))\cap(w + u_r^{-1}.P)\Bigr) \leq{}& \Bigl(\mathfrak{p}_r^{(k)}\Bigr)_{\ast} \mu^Q(R_{k}(w))\\ \ll_\epsilon{}& \frac{C\rho^\alpha}{\mu(Q)} \prod_{j = 0}^{k - 1}e^{-(n - j){t}} e^{\epsilon {t}}.
    \end{align*}
    The last inequality follows from the fact that we can cover $R_{k}(w)$ with $O(\prod_{j = 0}^{k - 1}e^{j{t}})$ many $e^{-n{t}}$ balls. 

    Let $F_r^{(k)} = \sqcup_{Q \in \mathcal{D}_{\rho}(r) } F_Q(r)$, we have
    \begin{align}\label{eqn:dim improve 1}
        \begin{aligned}
            \mu(F_r^{(k)} \cap u_r^{-1}a_{{t}}^{-1}B(x, \delta)) \ll_{\epsilon}{}&  C(e^{\frac{n}{2}{t}}\delta)^\alpha e^{-\frac{(2n - k + 1)k}{2}{t}}\delta^{-\epsilon}\\
        ={}& Ce^{-\frac{{t}}{2}((2n - \lfloor\alpha\rfloor + 1)\lfloor\alpha\rfloor - n\alpha)}\delta^{\alpha - \epsilon}.
        \end{aligned}
    \end{align}

    Now we apply Theorem~\ref{thm: proj GGW} with $\{\mu^Q\}_Q$, $\mathfrak{p}^{(k + 1)}_r$, scale $e^{-n{t}}$, and $\epsilon/n$. Using Fubini's theorem, again, there exists $\mathcal{E}_{{t}, \delta}' \subset [0, 1]$ with $|\mathcal{E}_{{t}, \delta}'| \ll_{\alpha, \epsilon} e^{-\star\epsilon^2{t}}$ so that for all $r \in [0, 1] \setminus \mathcal{E}_{{t}, \delta}'$, there exists $\mathcal{D}_{\rho}(r)'$ with 
    \begin{align*}
    \sum_{Q \notin \mathcal{D}_{\rho}(r)'} \mu(Q) \ll_{\alpha, \epsilon} e^{-\star\epsilon^2{t}}.
    \end{align*}
    Moreover, for all $Q \in \mathcal{D}_{\rho}(r)'$, there exists $F_{Q}(r)' \subset F_Q$ with 
    \begin{align*}
        \mu_Q(F_Q \setminus F_{Q}(r)') \ll_{\alpha, \epsilon} e^{-\star\epsilon^2{t}}
    \end{align*}
    so that for all $w \in V$, we have
    \begin{align*}
        \mu^Q (\Hom_Q(F_{Q}(r)')\cap(w + u_r^{-1}.P)) \leq{}& (\mathfrak{p}_r^{(k + 1)})_{\ast} \mu^Q(R_{k + 1}(w))\\ \ll_\epsilon{}& \frac{C\rho^\alpha}{\mu(Q)} \Biggl(\prod_{j = 0}^{k}e^{j{t}}\Biggr) e^{-n\alpha{t}} e^{\epsilon{t}}.
    \end{align*}
    As before, the last inequality follows from the fact that we can cover $R_{k+1}(w)$ via $O(\prod_{j = 0}^{k}e^{j{t}})$ many $e^{-n{t}}$ balls. 
    
    Let $F_r^{(k + 1)} = \sqcup_{Q \in \mathcal{D}_{\rho}(r) } F_Q(r)$, we have
    \begin{align}\label{eqn:dim improve 2}
        \begin{aligned}
            \mu(F_r^{(k + 1)} \cap u_r^{-1}a_{{t}}^{-1}B(x, \delta)) \ll_\epsilon{}& C(e^{\frac{n}{2}{t}}\delta)^\alpha e^{\frac{k(k + 1)}{2}{t}} e^{-n\alpha{t}} e^{\epsilon{t}}\\
        ={}& C e^{-(n\alpha - \lfloor\alpha\rfloor(\lfloor\alpha\rfloor + 1))\frac{{t}}{2}}\delta^{\alpha - \epsilon}.
        \end{aligned}
    \end{align}

    Let $\mathcal{E} = \mathcal{E}_{{t}, \delta} \cup \mathcal{E}_{{t}, \delta}'$ and for all $r \notin \mathcal{E}$, let $F_{t, \delta, r} = F_r^{(k)} \cap F_r^{(k + 1)}$. We have
    \begin{align*}
        |\mathcal{E}| \ll_{\alpha, \epsilon} e^{-\star\epsilon^2{t}}
    \end{align*}
    and
    \begin{align*}
        \mu(F \setminus F_r) \ll_{\alpha, \epsilon} e^{-\star\epsilon^2{t}}.
    \end{align*}
    Combining \eqref{eqn:dim improve 1} and \eqref{eqn:dim improve 2}, the theorem follows. 
\end{proof}

\bibliographystyle{halpha}
\bibliography{papers}

\end{document}